\documentclass[11pt]{amsart}
\usepackage[margin=1in]{geometry}
\usepackage{amsmath, amsfonts, amssymb, amsthm}
\usepackage[latin1]{inputenc}
\usepackage[T1]{fontenc}
\usepackage{hyperref}
\usepackage{color}
\usepackage{dsfont}
\usepackage{csquotes}
\usepackage{epigraph}
\usepackage{comment}
\usepackage{mlmodern}

\usepackage{color}
\definecolor{darkblue}{rgb}{0.0,0.0,0.3}
\hypersetup{
    colorlinks=false,
    linkcolor=blue,
    urlcolor=darkblue,
    }
  
\newcommand{\bH}{\mathbb{H}}
\newcommand{\bR}{\mathbb{R}}

\newcommand{\cS}{\mathcal{S}}
\newcommand{\Garding}{G\r{a}rding}
\newtheorem{prop}{Proposition}[section]
\newtheorem{theorem}[prop]{Theorem}
\newtheorem{remark}[prop]{Remark}
\newtheorem{lemma}[prop]{Lemma}
\newtheorem{corollary}[prop]{Corollary}

\DeclareMathOperator{\graph}{graph}

\begin{document}
  \title[Curvature estimates]{Curvature estimates for hypersurfaces of constant curvature in hyperbolic space}
  \author[B.~Wang]{Bin Wang}
  \address{Department of Mathematics and Statistics, McMaster University, 1280 Main Street West, Hamilton, ON, L8S 4K1, Canada.}
  
  \curraddr{Department of Mathematics \\ The Chinese University of Hong Kong\\ Shatin, N.T., Hong Kong.}
  
\email{bwang@math.cuhk.edu.hk}
  \begin{abstract}
  In this note, we prove that for every $0<\sigma<1$, there exists a smooth complete hypersurface $\Sigma$ in $\bH^{n+1}$ with prescribed asymptotic boundary $\partial \Sigma=\Gamma$ at infinity, whose principal curvatures $\kappa=(\kappa_1,\ldots,\kappa_n)$ lie in a general cone $K$ and satisfy $f(\kappa)=\sigma$ at each point of $\Sigma$. Previously, the problem has been studied by Guan-Spruck in [J. Eur. Math. Soc. (JEMS) 12 (2010), no. 3, 797-817], and they proved the existence result for $\sigma_0<\sigma <1$, where $\sigma_0>0$. A major ingredient of our proof is a refined curvature estimate of theirs that is applicable when the curvature function $f(\kappa)$ has controllable partial derivatives, but it is adequate for our purpose; specifically, we solve the problem for $f=H_{k}/H_{k-1}$ in the $k$-th \Garding\ cone where $H_k$ is the normalized $k$-th elementary symmetric polynomial and $1 \leq k \leq n$.
  \end{abstract}
  \maketitle

  \section{Introduction}
Fix $n \geq 2$; let $\bH^{n+1}$ denote the hyperbolic space of dimension $n+1$ and let $\partial_{\infty} \bH^{n+1}$ denote the ideal boundary of $\bH^{n+1}$ at infinity. Suppose that $f: K \to \bR$ is a smooth symmetric function defined in an open symmetric convex cone $K\subseteq \bR^n$ with vertex at the origin, containing the positive cone
\[K_{n}^{+}=\{\lambda \in \bR^n : \lambda_i>0\ \forall\ i\} \subseteq K. \]
Given a disjoint collection of closed embedded smooth $(n-1)$-dimensional submanifolds $\Gamma=\{\Gamma_1,\ldots,\Gamma_m\} \subseteq \partial_{\infty} \bH^{n+1}$ and a constant $0<\sigma<1$, we study the problem of finding a smooth complete hypersurface $\Sigma$ in $\bH^{n+1}$ satisfying 
\begin{equation}
\text{$\kappa[\Sigma] \in K$ and $f(\kappa[\Sigma])=\sigma$ at each point of $\Sigma$} \label{req1}
\end{equation}
with the asymptotic boundary
\begin{equation}
\partial \Sigma=\Gamma \label{req2}
\end{equation} where $\kappa[\Sigma]=(\kappa_1,\ldots,\kappa_n)$ denotes the vector of induced hyperbolic principal curvatures of $\Sigma$. Following \cite{JDG}, we may refer to it as the asymptotic Plateau problem in hyperbolic space and it will be investigated under a few standard assumptions on $f(\kappa)$; see section 2.

When $K=K_{n}^{+}$ is the positive cone, the problem (\ref{req1})-(\ref{req2}) has been completely solved by Bo Guan, Joel Spruck, Marek Szapiel and Ling Xiao in a series of papers \cite{JGA,SGAR,JDG}; they not only proved an existence result that is essentially optimal, but also established some uniqueness results under common geometrical assumptions about $\Gamma$, or when $f(\kappa)$ satisfies a certain property in a subset of $K$.

It is then natural to wonder whether the existence result holds in a general cone $K$. In \cite{JEMS}, Guan and Spruck have almost solved the problem in its full generality; they showed that if $\Gamma$ is mean-convex \footnote{$\Gamma$ is mean-convex if its Euclidean mean curvature is non-negative.} and $f(\kappa)$ satisfies (\ref{f cond 1})-(\ref{f cond extra}) in $K$, then for every $\sigma \in (\sigma_0,1)$ there exists a solution to (\ref{req1})-(\ref{req2}). Here $\sigma_0>0$ is strictly positive, whose value can be found by numerical computation to lie between 0.3703 and 0.3704.

The next task is certainly to either generalize their result for all $\sigma \in (0,1)$, or remove the mean-convexity condition on $\Gamma$. However, since the mean-convexity condition is needed in the derivation of gradient estimates \cite[Proposition 4.1]{JEMS}, we will keep it in this paper and try to extend their existence theorem \cite[Theorem 1.2]{JEMS} so that it will hold for all $\sigma \in (0,1)$.

As pointed out by Guan-Spruck themselves in \cite{JEMS}, the only issue occurs in the derivation of a maximum principle for the largest hyperbolic principal curvature i.e. $\kappa_{\max} \leq C$, which holds only for $\sigma \in (\sigma_0,1)$. Therefore, our task reduces to improving the curvature estimate. By examining their proof, we observe that a key to the desired estimation is to control the magnitude of $f_i$ or $\sum_{i=1}^{n} f_i$ in $K$ where $f_i:=\frac{\partial f}{\partial \kappa_i}$ is the $i$-th partial derivative of $f$, and so we prove

\begin{theorem} \label{my theorem}
Suppose $\Gamma$ is mean-convex i.e. its Euclidean mean curvature is non-negative. If the curvature function $f: K \to \bR$ satisfies \textbf{either} of the following conditions:
\begin{gather}
\text{there exists some $C>0$ such that $\sum_{i=1}^{n} f_i(\kappa) \leq C$ for all $\kappa \in K$} \tag{A}\label{my assumption1}\\
\text{there exists some $C>0$ such that if $\kappa_i > 0$ then $f_i \leq C\cdot\frac{f}{\kappa_i}$ in $K$} \tag{B}\label{my assumption2}
\end{gather}
in addition to (\ref{f cond 1})-(\ref{f cond extra}) in the general cone $K$, then for all $\sigma \in (0,1)$ there exists a solution to the asymptotic Plateau problem (\ref{req1})-(\ref{req2}).
\end{theorem}

To justify the imposition of (\ref{my assumption1}) and (\ref{my assumption2}), we note that they are satisfied by an important class of curvature functions: Let $H_k$ be the $k$-th normalized elementary symmetric polynomial and let $K_k$ denote the $k$-th \Garding\ cone, which are defined as 
\begin{align*}
H_k(\kappa_1,\ldots,\kappa_n)&:=\frac{1}{\binom{n}{k}} \sum_{1 \leq j_1 < j_2 < \cdots < j_k \leq n} \kappa_{j_1}\kappa_{j_2}\cdots \kappa_{j_k}, \quad 1 \leq k \leq n \\
K_k&:=\{\lambda \in \bR^{n}: H_j(\lambda)>0 \quad \forall\ 1 \leq j \leq k\}.\\
\end{align*} Then we can apply theorem \ref{my theorem} to obtain

\begin{corollary} \label{my corollary}
Suppose $\Gamma$ is mean-convex. For the following $(f,K)$ pairs:
\begin{enumerate}
\item[(i)] $f=\frac{H_{k}}{H_{k-1}}$ and $K=K_k$, $1\leq k \leq n$.
\item[(ii)] $f=(H_k/H_l)^{\frac{1}{k-l}}$ and $K=K_{k+1}$, $1 \leq l<k\leq n$.
\item[(iii)] $f=H_{k}^{1/k}$ and $K=K_{k+1}$, $1 \leq k \leq n$.
\end{enumerate}, the asymptotic Plateau problem (\ref{req1})-(\ref{req2}) is solvable for all $\sigma \in (0,1)$.
\end{corollary} 

Note that 
\[K_{n}^{+}=K_n \subseteq \cdots \subseteq K_{k+1}\subseteq K_k \subseteq K_{k-1} \subseteq \cdots \subseteq K_1.\]

The ultimate goal is to solve (\ref{req1})-(\ref{req2}) in the $k$-th Garding cone $K_k$ for either $f=H_{k}^{1/k}$ or $f=(H_k/H_l)^{\frac{1}{k-l}}$. Hence, our results in the above corollary are approaching this goal.

\begin{remark}
In fact, the case (i) was the original motivation for imposing the  assumption (\ref{my assumption1}). We later discovered that case (ii) and case (iii) have a similar property, which is (\ref{my assumption2}).
\end{remark}

\begin{remark}
According to \cite[Example 1.3]{Spruck-MSRI}, the curvature quotient 
\[\frac{H_n}{H_{n-1}}=\frac{n}{\sum_{i=1}^{n} \frac{1}{\kappa_i}}\] is called the \textbf{harmonic curvature}.

Therefore, the curvature quotient $\frac{H_k}{H_{k-1}} (1 \leq k \leq n)$ we considered here may be seen as a generalized version of the harmonic curvature. 

This particular curvature function has been studied in the context of curvature flows by Gerhardt \cite{Gerhardt-1990} and Urbas \cite{Urbas}, which inspired several proofs of geometric inequalities; to name a few, see \cite{Guan-Li, Li-Wei-Xiong, Ge-Wang-Wu} and references therein. For more recent works concerning this particular curvature function, see, for example, \cite{Lu-CAG, Andrews-Hu-Li}.
\end{remark}

\begin{remark}
By noting that $\frac{\partial H_1}{\partial \kappa_i}=\frac{1}{n}$ and $\frac{\partial H_n}{\partial \kappa_i}=\frac{H_n}{\kappa_i}$, it follows that theorem \ref{my theorem} also provides alternative proofs to the existence theorems for $(f,K)=(H_1,K_1)$ and $(f,K)=(H_n,K_n)$.
\end{remark}

This note is organized as follows. In section 2, we first briefly describe the method of solution for the asymptotic problem (\ref{req1})-(\ref{req2}) as exhibited in \cite{JGA,JEMS,SGAR,JDG}, then we list formulas and facts that will be used in later sections with references to their proofs. In section 3, we adapt the method from section 4 in \cite{SGAR} to improve the curvature estimate in \cite{JEMS}. With the aid of our assumptions (\ref{my assumption1}) and (\ref{my assumption2}), the curvature estimate will hold for all $\sigma \in (0,1)$ which then yields theorem \ref{my theorem} and hence partially extends the main theorem in \cite{JEMS}. We emphasize that it is here and only here in \cite{JEMS} that we modify, everything else has already been perfectly established and will remain intact. Finally in section 4, we prove corollary \ref{my corollary}. For more information about the asymptotic Plateau problem in hyperbolic space, see \cite{Rosenberg-Spruck,Nelli-Spruck,AJM,JGA,JEMS,SGAR,JDG} and the references therein.

After this work was completed, we learned that Lu \cite{Lu} solved (\ref{req1})-(\ref{req2}) for $f=H_{n-1}^{1/(n-1)}$ and $K=K_{n-1}$ by a different derivation for the curvature estimate, which employs a concavity inequality of Ren-Wang \cite{Ren-Wang-1}. Another new paper concerning the asymptotic Plateau problem is \cite{Sui-Sun-CPAA}.

\begin{remark}
In fact, Ren-Wang's inequality holds also for $f=H_{n-2}^{1/(n-2)}$ and $K=K_{n-2}$ in dimension $n \geq 5$; see \cite{Ren-Wang-2}. Therefore, Lu's work \cite{Lu} entails that \eqref{req1}-\eqref{req2} is solvable for $f=H_{n-2}^{1/(n-2)}$ and $K=K_{n-2}$ in dimension $n \geq 5$. However, Ren-Wang's inequality for $H_{n-2}^{1/(n-2)}$ does not hold in dimension 4; see the counterexample at the end of section 3 in \cite{Ren-Wang-1}. Note that this case corresponds to the constant scalar curvature equation $H_{2}^{1/2}=\sigma$ in dimension 4, which still remains open for $\sigma \in (0,\sigma_0]$.
\end{remark}

\section{Preliminaries}
We will use the upper half-space model for the hyperbolic space 
\[\bH^{n+1}=\{(x,x_{n+1}) \in \bR^{n+1}: x_{n+1}>0\}\] equipped with the hyperbolic metric
\[ds^2=\frac{\sum_{i=1}^{n+1}dx_{i}^2}{x_{n+1}^2}\] so that we can identify $\partial_{\infty}\bH^{n+1}$ with $\bR^{n}=\bR^{n} \times \{0\} \subseteq \bR^{n+1}$ and (\ref{req2}) can be understood in the Euclidean sense. We say $\Sigma$ has compact asymptotic boundary if $\partial \Sigma \subseteq \partial_{\infty} \bH^{n+1}$ is compact with respect to the Euclidean metric in $\bR^n$.

Throughout this note all hypersurfaces in $\bH^{n+1}$ are assumed to be connected and orientable. If $\Sigma$ is a complete hypersurface in $\bH^{n+1}$ with compact asymptotic boundary at infinity, then the normal vector field of $\Sigma$ is chosen to be the one pointing toward the unique unbounded region in $\bR_{+}^{n+1} \setminus \Sigma$, and both the hyperbolic and Euclidean principal curvatures of $\Sigma$ are calculated with respect to this normal vector field.

Suppose $\Sigma$ is locally represented as a graph of a positive $C^2$ function $u(x)>0$ over a domain $\Omega \subseteq \bR^n$:
\[\Sigma=\{(x,u(x)) \in \bR^{n+1}: x \in \Omega\}\] oriented by the upward Euclidean unit normal vector field $\nu$ to $\Sigma$:
\[\nu=\left(-\frac{Du}{w},\frac{1}{w}\right),\ \text{where $w=\sqrt{1+|Du|^2}$.}\]

Let $\cS$ denote the space of $n \times n$ symmetric matrices and \[\cS_K:=\{A \in \cS: \lambda(A) \in K\}\] where $\lambda(A)=(\lambda_1,\ldots,\lambda_n)$ denotes the eigenvalues of $A=[a_{ij}]$. We define a function \[F: \cS_K \to \bR\] by $F(A):=f(\lambda(A))$ and denote 
\[F^{ij}:=\frac{\partial F}{\partial a_{ij}}, \quad F^{ij,kl}:=\frac{\partial^2F}{\partial a_{ij}\partial a_{kl}}\] The curvature relation (\ref{req1}) can thus be written locally as 
\[F(h_{ij})=\sigma\] where $h_{ij}$ is the hyperbolic second fundamental form of $\Sigma$. 

According to the method exhibited in \cite{JGA,JEMS,SGAR,JDG}), a solution to (\ref{req1})-(\ref{req2}) can then be constructed as the graph $\Sigma=\graph[u]$ of a smooth solution $u(x)$ to the following Dirchlet problem 
\begin{align*}
G(D^2u,Du,u)=\sigma, \quad u>0 \quad  &\text{in $\Omega$}\\
u=0 \quad &\text{on $\partial \Omega$}
\end{align*} with $\kappa(\graph[u]) \in K$. The existence of solutions can be ensured by the method of continuity which requires $C^{2,\alpha}$ estimates for admissible solutions. As we shall see below, we will impose conditions on the curvature function $f(\kappa)$ so that the operator $G$ is elliptic and concave in $D^2u$; see (\ref{f cond 1}) and (\ref{concavity}). Also, the condition (\ref{f cond 3}) will imply that $G$ is uniformly elliptic on compact subdomains of $\Omega$ and thus by the Evans-Krylov theorem \cite{Evans,Krylov} it suffices to derive \textit{a priori} $C^2$ estimates for admissible solutions and this is where the curvature estimate $\kappa_{\max} \leq C$ enters into place. We will provide a detailed proof for the curvature estimate in section 3.

We now list a few geometric identities that will be used in deriving the curvature estimate of section 3. For a hypersurface $\Sigma$ in $\bH^{n+1}$, let $g$ and $\nabla$ denote the induced hyperbolic metric and Levi-Civita connection on $\Sigma$, respectively. Viewing $\Sigma$ as a submanifold of $\bR^{n+1}$, let $\tilde{g}$ denote the induced metric on $\Sigma$ from $\bR^{n+1}$ and $\tilde{\nabla}$ is its corresponding Levi-Civita connection. We have

\begin{lemma} \label{formulas}
In a local orthonormal frame $\{\tau_1,\ldots,\tau_n\}$, we have 
\begin{enumerate}
\item[(i)] $\sum_{i=1}^{n} \frac{u_{i}^2}{u^2}=|\tilde{\nabla}u|^2=1-(\nu^{n+1})^2 \leq 1$.
\item[(ii)] $\nabla_i\nu^{n+1}=-\frac{u_i}{u}(\kappa_i-\nu^{n+1})$.
\item[(iii)] $F^{ii}h_{11ii}=-F^{ij,kl}h_{ij1}h_{rs1}+\sigma(1+\kappa_{1}^2)-\kappa_1\left(\sum f_i+\sum \kappa_{i}^2f_i\right)$.
\item[(iv)] $F^{ij}\nabla_{ij}\nu^{n+1}=\frac{2u_i}{u}F^{ij}\nabla_{j}\nu^{n+1}+\sigma[1+(\nu^{n+1})^2]-\nu^{n+1}\left(\sum f_i + \sum \kappa_{i}^2f_i\right)$.
\end{enumerate}
\end{lemma}
\begin{proof}
For (i) and (ii), see lemma 4.1 in \cite{Sui}. For (iii) and (iv), see lemma 4.2 and lemma 4.3 in \cite{SGAR}.
\end{proof}

Now let us specify the assumptions on the curvature function $f(\kappa) \in C^2 (K) \cap  C( \overline{K})$. First, $f(\kappa)$ is assumed to satisfy the following standard assumptions:
\begin{gather}
f_i(\lambda):=\frac{\partial f}{\partial \kappa_i}(\lambda)>0 \quad \text{for $\lambda \in K$ and $1 \leq i \leq n$} \label{f cond 1} \\
\text{$f$ is a concave function in $K$} \label{concavity}\\
\text{$f>0$ in $K$ and $f=0$ on $\partial K$.} \label{f cond 3}
\end{gather} We shall also impose for convenience that
\begin{gather}
\text{$f$ is normalized: $f(1,\ldots,1)=1$,} \label{normalized}\\
\text{and $f$ is homogeneous of degree one in $K$} \label{f cond last} 
\end{gather}; and assume a technical condition 
\begin{gather}
\lim\limits_{R \to \infty} f(\lambda_1,\ldots,\lambda_{n-1},\lambda_n+R) \geq 1+\varepsilon_0 \quad \text{uniformly in $B_{\delta_0}(\mathbf{1})$} \label{f cond extra}
\end{gather} for some fixed $\varepsilon_0>0$ and $\delta_0>0$, which is needed in deriving a pure normal second derivative estimate (see section 5 in \cite{JEMS}). This condition can be removed if $K=K_{n}^{+}$, as shown in \cite{JDG}.

All these conditions are satisfied by the class of curvature functions that we care the most about; namely, the higher order mean curvatures $H_{k}^{1/k}$ where $1 \leq k \leq n$, and the class of curvature quotients $(H_k/H_l)^{\frac{1}{k-l}}$ where $1 \leq l<k \leq n$.
\begin{lemma} \label{curvature function condition}
Conditions (\ref{f cond 1})-(\ref{f cond extra}) are satisfied by both of 
\begin{enumerate}
\item[(i)] $H_{k}^{1/k}$ for $1 \leq k \leq n$, and
\item[(ii)] $(H_k/H_l)^{\frac{1}{k-l}}$ for $1 \leq l<k \leq n$
\end{enumerate} in the $k$-th Garding cone $K_k$.
\end{lemma}
\begin{proof}
For (\ref{f cond 1}) and (\ref{concavity}), see theorem 15.17 and theorem 15.18 in \cite{Lieberman}. Conditions (\ref{f cond 3})-(\ref{f cond last}) follow from the definitions of $H_k$ and $K_k$. For (\ref{f cond extra}), see the discussion preceding theorem 1.2 in \cite{JEMS}.
\end{proof}

The following will also be useful in section 3.
\begin{lemma}\label{AG}
Let $F$ and $f$ be defined as above. Then for all $A=[a_{ij}] \in \cS$, we have 
\begin{enumerate}
\item[(i)] $F^{ij,kl}a_{ij}a_{kl}=f_{ij}a_{ii}a_{jj}+\sum_{i \neq j} \frac{f_i-f_j}{\kappa_i-\kappa_j}(a_{ij})^2$ and 
\item[(ii)] $\frac{f_i-f_j}{\kappa_i-\kappa_j} \leq 0$ for $i \neq j$.
\end{enumerate}
\end{lemma}
\begin{proof}
See lemma 1.1 in \cite{Gerhardt} or section 2 in \cite{Andrews}.
\end{proof}

\section{The Curvature Estimate}
\begin{theorem}
Suppose that $f (\kappa)$ satisfies (\ref{my assumption1}) or (\ref{my assumption2}) in addition to (\ref{f cond 1})-(\ref{f cond extra}) in $K$. If $\Sigma=\graph(u)$ is a smooth vertical graph in $\bH^{n+1}$ satisfying (\ref{req1})-(\ref{req2}) for a given $\sigma \in (0,1)$, then 
\[\max_{x \in \Sigma} \kappa_{\max}(x) \leq C(1+\max_{x \in \partial \Sigma} \kappa_{\max}(x))\] where $\kappa_{\max}:=\max\limits_{1 \leq i \leq n} \kappa_i$ is the maximal principal curvature and $C>0$ is a controllable constant depending on $\sigma$.
\end{theorem}
\begin{proof}
Let
\[M_{0}:=\sup_{ \Sigma } \frac{\kappa_{\max}}{\nu^{n+1}-a}\] where $a>0$ is a constant depending on $\sigma$ such that $\nu^{n+1} \geq 2a>0$; such a constant can always be chosen because of the gradient estimate $\nu^{n+1} \geq \sigma > 0$; see proposition 4.1 in \cite{JEMS}.

Suppose the maximum $M_0$ is attained at an interior point $x_0 \in \Sigma$. We choose a local orthonormal frame around $x_0$ such that $h_{ij}(x_0)=\kappa_i(x_0) \delta_{ij}$ and for notational convenience, we may assume \[\kappa_{\max}(x_0)=\kappa_1(x_0) \geq \kappa_2(x_0) \geq \cdots \geq \kappa_n(x_0).\] In what follows, we will suppress notation to not write out $x_0$ but keep in mind that all the calculations are done at $x_0$. Note that we may also assume $\kappa_1>\nu^{n+1}>0$ otherwise we would have $\kappa_1 \leq \nu^{n+1} \leq 1$ already.

Now since $\log \frac{h_{11}}{\nu^{n+1}-a}$ has a local maximum at $x_0$, we have 

\begin{align}
&\frac{h_{11i}}{h_{11}}-\frac{\nabla_i\nu^{n+1}}{\nu^{n+1}-a}=0 \label{1st derivative=0} \\
&\frac{h_{11ii}}{h_{11}}-\frac{\nabla_{ii}\nu^{n+1}}{\nu^{n+1}-a} \leq 0 \label{2nd derivative<=0}
\end{align}. Next, we multiply (\ref{2nd derivative<=0}) by $F^{ii}h_{11}=\kappa_1F^{ii}$ (and summing over $i$) to get 
\begin{equation}
F^{ii}h_{11ii}-\frac{\kappa_1}{\nu^{n+1}-a} F^{ii}\nabla_{ii}\nu^{n+1} \leq 0 \label{the inequality 1}
\end{equation}. By expanding each term in (\ref{the inequality 1}), we will show an inequality of the form
\[\text{$\kappa_{1}^2-C\kappa_1\leq 0$}\] for some $C>0$ depending on $\sigma$.

First, we use lemma \ref{formulas} (ii) and (\ref{1st derivative=0}) to obtain
\[h_{11i}=\frac{\kappa_1}{\nu^{n+1}-a}\nabla_i\nu^{n+1}=-\frac{\kappa_1}{\nu^{n+1}-a}\frac{u_i}{u}(\kappa_i-\nu^{n+1})\]

Then, by lemma \ref{formulas} (ii)-(iv) and lemma \ref{AG} (i), we have
\begin{align*}
F^{ii}h_{11ii}\geq &\ 2\sum_{i \geq 2} \frac{f_i-f_1}{\kappa_1-\kappa_i} h_{i11}^2+\sigma(1+\kappa_{1}^2)-\kappa_1\left(\sum_{i=1}^{n} f_i + \sum_{i=1}^{n} \kappa_{i}^2 f_i\right) \\
\geq & \ 2\kappa_{1}^2\sum_{i \geq 2} \frac{f_i-f_1}{\kappa_1-\kappa_i} \frac{u_{i}^2}{u^2} \left(\frac{\kappa_i-\nu^{n+1}}{\nu^{n+1}-a}\right)^2 +\sigma\kappa_{1}^2\\
&-\kappa_1\left(\sum_{i=1}^{n} f_i + \sum_{i=1}^{n} \kappa_{i}^2 f_i\right)\\
\end{align*} and 
\begin{align*}
-\frac{\kappa_1}{\nu^{n+1}-a} F^{ii}\nabla_{ii}\nu^{n+1}=&-\frac{2\kappa_1}{\nu^{n+1}-a} f_{i}\frac{u_i}{u}\nabla_i\nu^{n+1}-\sigma \frac{1+(\nu^{n+1})^2}{\nu^{n+1}-a}\kappa_1 \\
&+\frac{\nu^{n+1}\kappa_1}{\nu^{n+1}-a}\left(\sum_{i=1}^{n} f_i+\sum_{i=1}^{n} \kappa_{i}^2 f_i\right) \\
\geq &\ 2\kappa_1 \sum_{i=1}^{n} f_{i}\frac{u_{i}^2}{u^2}\frac{\kappa_i-\nu^{n+1}}{\nu^{n+1}-a}-\frac{2\sigma}{a}\kappa_1 \\
&+\frac{\nu^{n+1}\kappa_1}{\nu^{n+1}-a}\left(\sum_{i=1}^{n} f_i+\sum_{i=1}^{n} \kappa_{i}^2 f_i\right) \\
\end{align*}
Adding them up, the oringial inequality (\ref{the inequality 1}) becomes

\begin{align}
0 \geq &\ \sigma \left(\kappa_{1}^2-\frac{2}{a}\kappa_1\right) +\frac{a \kappa_1}{\nu^{n+1}-a} \left(\sum_{i=1}^{n} f_i + \sum_{i=1}^{n} \kappa_{i}^2f_i\right) \label{first line}\\
&+2\kappa_{1}^2\sum_{i \geq 2} \frac{f_i-f_1}{\kappa_1-\kappa_i}\frac{u_{i}^2}{u^2} \left(\frac{\kappa_i-\nu^{n+1}}{\nu^{n+1}-a}\right)^2+2\kappa_1\sum_{i=1}^{n} f_i \frac{u_{i}^2}{u^2} \frac{\kappa_i-\nu^{n+1}}{\nu^{n+1}-a} \label{second line}
\end{align} By (\ref{f cond 1}), the first line (\ref{first line}) can be made positive if we assume $\kappa_1 \geq \frac{2}{a}$. The first term in the second line (\ref{second line}) is positive due to lemma \ref{AG} (ii); only the second term can be potentially negative if $\kappa_i<\nu^{n+1}$.

Note that by lemma \ref{formulas} (i), we have
\begin{align*}
&\frac{a\kappa_1}{\nu^{n+1}-a} \sum_{i=1}^{n} \kappa_{i}^2f_i+2\kappa_1\sum_{i=1}^{n} f_i \frac{u_{i}^2}{u^2} \frac{\kappa_i-\nu^{n+1}}{\nu^{n+1}-a} \\
\geq &\kappa_1\sum_{\kappa_i < \nu^{n+1}} \frac{f_i}{\nu^{n+1}-a} a\kappa_{i}^2+\kappa_1\sum_{\kappa_i<\nu^{n+1}} \frac{f_i}{\nu^{n+1}-a} 2(\kappa_i-\nu^{n+1}) \\
\geq &\frac{\kappa_1}{\nu^{n+1}-a} \sum_{\kappa_i<\nu^{n+1}} f_i(a\kappa_{i}^2+2\kappa_i-2)
\end{align*}, whose summand is positive if $\kappa_i\leq -\eta$ where $\eta:=\frac{1+\sqrt{1+2a}}{a}$. This motivates us to consider 
\[J=\{i: -\eta<\kappa_i<\nu^{n+1},\ \theta f_i<f_1\}, \quad L=\{i: -\eta<\kappa_i<\nu^{n+1},\ \theta f_i \geq f_1\}\] where $\theta \in (0,1)$ is to be determined later.

Now the second term in (\ref{second line}) can be split into three sums:
\begin{equation}
2\kappa_1\sum_{i=1}^{n} f_i \frac{u_{i}^2}{u^2} \frac{\kappa_i-\nu^{n+1}}{\nu^{n+1}-a}=2\kappa_1\left(\sum_{\kappa_i \leq -\eta}+\sum_{i \in J}+\sum_{i \in L}\right)f_i \frac{u_{i}^2}{u^2} \frac{\kappa_i-\nu^{n+1}}{\nu^{n+1}-a} \label{second term in second line}
\end{equation}

Also, the second term in (\ref{first line}) can be split as 
\begin{equation}
\begin{split}
&\frac{a \kappa_1}{\nu^{n+1}-a} \left(\sum_{i=1}^{n} f_i + \sum_{i=1}^{n} \kappa_{i}^2f_i\right) \\
\geq \ &\frac{a \kappa_1}{\nu^{n+1}-a} \left(\sum_{\kappa_i \leq -\eta} \kappa_{i}^2f_i+\sum_{i \in L} f_i+\sum_{i \in L} \kappa_{i}^2f_i\right)
\end{split} \label{sum of f_i}
\end{equation}

As we have argued above, the $\kappa_i \leq -\eta$ part of (\ref{second term in second line}) adds to the $\kappa_i \leq -\eta$ part of (\ref{sum of f_i}) which results in a positive quantity. While for the J-sum in (\ref{second term in second line}), we use lemma \ref{formulas} (i) to obtain
\begin{equation}
2\kappa_1\sum_{i \in J} f_i \frac{u_{i}^2}{u^2}\frac{\kappa_i-\nu^{n+1}}{\nu^{n+1}-a}\geq \frac{2\kappa_1}{\nu^{n+1}-a}\sum_{i \in J} f_i (\kappa_i-\nu^{n+1}) \geq -\frac{2\kappa_1}{a}(\eta+1)\sum_{i \in J} f_i \label{J sum}
\end{equation}
Before we proceed to estimate the L-sum in (\ref{second term in second line}), we note that
\[\frac{\kappa_1}{\kappa_1-\kappa_i}\geq \frac{1}{1+\lambda} \quad \text{for $i \in L$}\] where $\lambda:=\frac{\eta}{\kappa_1}>0$.
Then, we can use the first term in (\ref{second line}) so that
\begin{equation}
\begin{split}
&2\kappa_{1}^2\sum_{i \geq 2} \frac{f_i-f_1}{\kappa_1-\kappa_i}\frac{u_{i}^2}{u^2} \left(\frac{\kappa_i-\nu^{n+1}}{\nu^{n+1}-a}\right)^2+2\kappa_1\sum_{i \in L} f_i \frac{u_{i}^2}{u^2} \frac{\kappa_i-\nu^{n+1}}{\nu^{n+1}-a} \\
\geq \ & 2\kappa_{1}^2\sum_{i \in L} \frac{f_i-f_1}{\kappa_1-\kappa_i}\frac{u_{i}^2}{u^2} \left(\frac{\kappa_i-\nu^{n+1}}{\nu^{n+1}-a}\right)^2+2\kappa_1\sum_{i \in L} f_i \frac{u_{i}^2}{u^2} \frac{\kappa_i-\nu^{n+1}}{\nu^{n+1}-a} \\
\geq \ & 2\kappa_{1}\frac{1-\theta}{1+\lambda}\sum_{i \in L} f_i\frac{u_{i}^2}{u^2} \left(\frac{\kappa_i-\nu^{n+1}}{\nu^{n+1}-a}\right)^2+2\kappa_1\sum_{i \in L} f_i \frac{u_{i}^2}{u^2} \frac{\kappa_i-\nu^{n+1}}{\nu^{n+1}-a} \\
=\ & 2\kappa_1\sum_{i \in L} f_i \frac{u_{i}^2}{u^2}\frac{\kappa_i-\nu^{n+1}}{\nu^{n+1}-a}\left[(1-\mu)\frac{\kappa_i-\nu^{n+1}}{\nu^{n+1}-a}+1\right] \\
=\ &2\kappa_1\sum_{i \in L} f_i \frac{u_{i}^2}{u^2}\frac{(\kappa_i-\nu^{n+1})(\kappa_i-a)}{(\nu^{n+1}-a)^2}-2\mu \kappa_1\sum_{i \in L} f_i \frac{u_{i}^2}{u^2}\left(\frac{\kappa_i-\nu^{n+1}}{\nu^{n+1}-a}\right)^2 \\
\geq \ & -2\kappa_1\sum_{i \in L, \kappa_i>a} f_i-\frac{2\mu}{a}\frac{\kappa_1}{\nu^{n+1}-a} \sum_{i \in L} f_i(\kappa_i-\nu^{n+1})^2
\end{split} \label{L sum}
\end{equation} where we have used lemma \ref{formulas} (i) and $0<\mu:=\frac{\theta+\lambda}{1+\lambda}<1$.

Now, let us first deal with the second term in the last line of (\ref{L sum}); we add it to the L-part of (\ref{sum of f_i}):
\begin{align*}
&\frac{a\kappa_1}{\nu^{n+1}-a} \left(\sum_{i \in L} f_i+\sum_{i \in L} \kappa_{i}^2 f_i\right)-\frac{2\mu}{a}\frac{\kappa_1}{\nu^{n+1}-a} \sum_{i \in L} f_i(\kappa_i-\nu^{n+1})^2 \\
=&\frac{\kappa_1}{\nu^{n+1}-a}\sum_{i \in L}f_i \left[\left(a-\frac{2\mu}{a}\right)\kappa_{i}^2+\frac{4\mu}{a}\kappa_i\nu^{n+1}+a-\frac{2\mu}{a}(\nu^{n+1})^2\right]\\
\geq & \frac{\kappa_1}{\nu^{n+1}-a}\sum_{i \in L}f_i \left(\frac{a}{2}\kappa_{i}^2+a^2\kappa_i+\frac{a}{2}\right)\\
>&\ 0
\end{align*} where we have used $2a \leq \nu^{n+1} \leq 1$ and chosen $\frac{a^2}{8} \leq \mu \leq \frac{a^2}{4}$. The choice of $\mu$ is possible because it is equivalent to 
\[\frac{a^2}{8}+\lambda\left(\frac{a^2}{8}-1\right) \leq \theta \leq \frac{a^2}{4}+\lambda\left(\frac{a^2}{4}-1\right)\] and we just need to make sure the left hand side is strictly between 0 and 1:
\[0<\frac{a^2}{8}+\lambda\left(\frac{a^2}{8}-1\right)<1 \Longleftrightarrow  -1<\lambda < \frac{a^2}{8-a^2}\] which is indeed achievable if we assume 
\[\kappa_1>\eta \cdot \frac{8-a^2}{a^2}.\] Thus we can choose $\theta \in (0,1)$ such that 
\[0<\frac{a^2}{8}+\lambda\left(\frac{a^2}{8}-1\right)\leq \theta < \min \left\{1,\ \frac{a^2}{4}+\lambda\left(\frac{a^2}{4}-1\right)\right\}.\]

Finally, there are only two trouble terms remaining, namely, the J-sum (\ref{J sum}) and the first term in the last line of (\ref{L sum}):
\begin{equation}
-\frac{2\kappa_1}{a}(\eta+1)\sum_{i \in J} f_i-2\kappa_1\sum_{i \in L,\kappa_i>a} f_i\label{L sum 2}
\end{equation}. This is the only place where we apply one of the assumptions (\ref{my assumption1}) or (\ref{my assumption2}); and we recall them here:
\begin{align*}
\text{Assumption (\ref{my assumption1}):} \quad& \exists\ C>0\  \forall\ \kappa \in K\ \text{s.t.}\ \sum_{i=1}^{n} f_i(\kappa) \leq C\  \\
\text{Assumption (\ref{my assumption2}):} \quad &\exists\ C>0\ \forall\ \kappa \in K\ \text{s.t.}\ \kappa_i >0 \Rightarrow f_i \leq \frac{Cf}{\kappa_i}
\end{align*}

If $f$ satisfies (\ref{my assumption1}) in $K$, then the two trouble terms can be easily handled:
\[-\frac{2\kappa_1}{a}(\eta+1)\sum_{i \in J} f_i-2\kappa_1\sum_{i \in L,\kappa_i>a} f_i \geq -C\kappa_1.\]

If $f$ satisfies (\ref{my assumption2}) in $K$, then by (\ref{req1}) and $\kappa_1 \geq \frac{2}{a}$, each term can be estimated as follows:
\begin{align*}
-\frac{2\kappa_1}{a}(\eta+1)\sum_{i \in J} f_i&\geq -\frac{2n\kappa_1}{\theta a}(\eta+1)f_1\geq -\frac{\sigma n\kappa_1}{\theta }(\eta+1) \geq -C\kappa_1, \\
\text{and} \quad -2\kappa_1\sum_{i \in L,\kappa_i>a} f_i &\geq -2\kappa_1C\sum_{i \in L,\kappa_i>a} \frac{\sigma}{a} \geq -\frac{2\sigma n\kappa_1}{a}C\geq -C\kappa_1.
\end{align*} In either case, the inequality (\ref{first line})-(\ref{second line}) implies that 
\[0 \geq \kappa_{1}^2-C\kappa_1\] and the proof is complete.
\end{proof}

\section{Proof of Corollary \ref{my corollary}}
For convenience, let us recall the corollary here: If $\Gamma$ is mean-convex, then for the following $(f,K)$ pairs:
\begin{enumerate}
\item[(i)] $f=\frac{H_{k}}{H_{k-1}}$ and $K=K_k$, $1 \leq k \leq n$.
\item[(ii)] $f=(H_k/H_l)^{\frac{1}{k-l}}$ and $K=K_{k+1}$, $1 \leq l<k \leq n$.
\item[(iii)] $f=H_{k}^{1/k}$ and $K=K_{k+1}$, $1 \leq k \leq n$.
\end{enumerate}, the asymptotic Plateau problem (\ref{req1})-(\ref{req2}) is solvable for all $\sigma \in (0,1)$.

\begin{proof}
According to lemma \ref{curvature function condition}, it suffices to verify that the curvature functions satisfy (\ref{my assumption1}) or (\ref{my assumption2}) in the indicated cones. Before we proceed, let us first recall the following properties of the elementary symmetric polynomials:
\begin{align}
H_k&=\kappa_i \frac{\partial H_k}{\partial \kappa_i}+\frac{n-k}{k+1}\frac{\partial H_{k+1}}{\partial \kappa_i}, \label{H_k}\\
\text{and} \quad \sum_{i=1}^{n} &\frac{\partial H_k}{\partial \kappa_i}=kH_{k-1}. \label{derivative H_k}
\end{align} 

For (i), we verify that $f=\frac{H_k}{H_{k-1}}$ satisfies (\ref{my assumption1}) in $K_k$. Indeed, we have that 
\begin{align*}
\sum_{i=1}^{n} f_i &=\sum_{i=1}^{n} \frac{\partial}{\partial \kappa_i} \frac{H_k}{H_{k-1}}= \sum_{i=1}^{n} \frac{\frac{\partial H_k}{\partial \kappa_i}H_{k-1}-H_k\frac{\partial H_{k-1}}{\partial \kappa_i}}{H_{k-1}^2} \\
& \leq \sum_{i=1}^{n}\frac{\frac{\partial H_k}{\partial \kappa_i}H_{k-1}}{H_{k-1}^2} \quad \text{since $H_k>0$, $\frac{\partial H_{k-1}}{\partial \kappa_i}>0$ in $K_k$} \\
&=\frac{kH_{k-1}H_{k-1}}{H_{k-1}^2} \quad \text{by (\ref{derivative H_k})}\\
&=k
\end{align*} for all $\kappa \in K_k$.

For (ii) and (iii), suppose that $\kappa_i>0$ and by invoking (\ref{H_k}), we have that  
\begin{equation}
\frac{\partial H_{k}}{\partial \kappa_i}=\frac{1}{\kappa_i}\left(H_k-\underbrace{C_{n,k}\frac{\partial H_{k+1}}{\partial \kappa_i}}_{>0}\right) \leq \frac{H_k}{\kappa_i} \quad \text{in $\{\kappa_i>0\} \cap K_{k+1}$.} \label{partial H_k}
\end{equation}

If $f=(H_k/H_l)^{\frac{1}{k-l}}$, then for $\kappa \in \{\kappa_i>0\} \cap K_{k+1}$
\begin{align*}
f_i&=\frac{1}{k-l}(H_k/H_l)^{\frac{1}{k-l}-1} \frac{\frac{\partial H_k}{\partial \kappa_i} H_l-H_k\frac{\partial H_l}{\partial \kappa_i}}{H_{l}^2} \\
&\leq \frac{1}{k-l}(H_k/H_l)^{\frac{1}{k-l}-1} \frac{1}{H_l}\frac{\partial H_k}{\partial \kappa_i} \quad \text{since $H_k\frac{\partial H_l}{\partial \kappa_i}>0$ in $K_k$}\\
&\leq \frac{1}{k-l}(H_k/H_l)^{\frac{1}{k-l}-1} \frac{H_k}{H_l}\frac{1}{\kappa_i} \quad \text{by (\ref{partial H_k})}\\
&=\frac{f}{k-l}\frac{1}{\kappa_i}.
\end{align*}.

Similarly, if $f=H_{k}^{\frac{1}{k}}$, then for $\kappa \in \{\kappa_i>0\} \cap K_{k+1}$
\begin{align*}
f_i=\frac{1}{k}H_{k}^{\frac{1}{k}-1} \frac{\partial H_k}{\partial \kappa_i} \leq \frac{1}{k}H_{k}^{\frac{1}{k}-1} \frac{H_k}{\kappa_i}=\frac{f}{k} \frac{1}{\kappa_i}.
\end{align*} This shows that $H_{k}^{\frac{1}{k}}$ and $(H_k/H_l)^{\frac{1}{k-l}}$ both satisfy (\ref{my assumption2}) in $K_{k+1}$.
\end{proof}

\medskip

\noindent

\section*{Acknowledgements}
  
This note is an extension of the author's M.Sc. thesis at McMaster University and the author would like to thank Professor Siyuan Lu for his supervision on the thesis.
  
\newpage

\bibliographystyle{amsplain}
\bibliography{refs}

\end{document}